\documentclass[14pt]{extarticle}
\usepackage{extsizes}
\usepackage{amssymb,amsfonts,amsthm}
\usepackage{amsmath}
\usepackage[makeroom]{cancel}
\usepackage{mathtools}
\usepackage{mathrsfs,pifont}
\usepackage{slashed,mathabx} 
\usepackage[bbgreekl]{mathbbol}
\usepackage{ytableau} 
\usepackage{tikz}
\usetikzlibrary{calc}

\usepackage[top=1in, bottom=1.25in, left=0.75in,right=0.75in]{geometry}
\usepackage{authblk}

\usepackage[font={small,sl}]{caption}
\usepackage[all]{xy}
\usepackage{hyperref}

\newtheorem{theorem}{Theorem}[section]
\newtheorem{lemma}[theorem]{Lemma}
\newtheorem{corollary}[theorem]{Corollary}
\newtheorem{conjecture}{Conjecture}
\theoremstyle{definition}

\theoremstyle{remark}

\numberwithin{equation}{section}

\makeatletter
\def\blfootnote{\xdef\@thefnmark{}\@footnotetext}
\makeatother

\begin{document}

\title{Infinitesimal deformations of $\mathfrak{sl}_2$ with\\ a twisted Jacobi identity}
\author{Haoran Zhu\thanks{Division of Mathematical Sciences, Nanyang Technological University, 21 Nanyang Link, Singapore 637371. Email: zhuh0031@e.ntu.edu.sg} }

\date{}

\maketitle

\begin{abstract}


We show that whenever
\[
[\,\cdot,\cdot]_t = [\,\cdot,\cdot]_0 + t[\,\cdot,\cdot]_1,\qquad
\alpha_t = \mathrm{id} + t\alpha_1
\]
define an infinitesimal Hom--Lie deformation of $\mathfrak{sl}_2(\mathbb K)$ over $\mathbb K[t]/(t^2)$ and $(\mathfrak{sl}_2(\mathbb K),[\,\cdot,\cdot]_0,\alpha_1)$ is a Hom--Lie algebra, then the deformed bracket $[\,\cdot,\cdot]_t$ satisfies the ordinary Jacobi identity over $\mathbb K[t]$. This solves a conjecture of Makhlouf and Silvestrov from 2010.

\end{abstract}

    

\section{Introduction}

Throughout, we let $\mathbb K$ be a field of characteristic $0$. A \textbf{Hom--Lie algebra} over $\mathbb K$ is a triple $(V,[\cdot,\cdot],\alpha)$
consisting of a $\mathbb K$--vector space $V$, a bilinear skew--symmetric bracket
\[
[\cdot,\cdot]\colon V\times V\longrightarrow V,
\]
and a linear map $\alpha\colon V\to V$ such that the Hom--Jacobi identity
\[
\sum_{\circlearrowleft x,y,z} [\alpha(x),[y,z]]=0
\]
holds for all $x,y,z\in V$, where the sum is taken over the cyclic permutations of $(x,y,z)$.

Hom--Lie algebras were introduced by Hartwig, Larsson, and Silvestrov~\cite{HLS} as twisted analogues of Lie algebras, motivated in particular by $q$--deformations and discretisations of Lie algebras of vector fields and related constructions in the framework of quasi--Lie algebras; see \cite{MakSil-Forum,MakSil-HomAlg} and the references therein. This notion has since been generalised in many directions, giving rise to a variety of Hom-type structures such as Hom--associative~\cite{MakSil-HomAlg}, Hom--Hopf~\cite{CG}, Hom--Novikov~\cite{Y2}, BiHom--Lie algebras~\cite{GMM}, and Hom-quantum groups~\cite{Y1}.

Following the approach of Gerstenhaber~\cite{Ge} and of Nijenhuis--Richardson~\cite{NR} for associative and Lie algebras, Makhlouf and Silvestrov~\cite{MakSil-Forum} developed a theory of one--parameter formal deformations for Hom--associative and Hom--Lie algebras. In the Hom--Lie setting, a one--parameter formal deformation of a fixed Hom--Lie algebra $(V,[\cdot,\cdot]_0,\alpha_0)$ is given by formal power series
\[
[\cdot,\cdot]_t = \sum_{i\ge 0} t^i [\cdot,\cdot]_i,
\qquad
\alpha_t = \sum_{i\ge 0} t^i \alpha_i
\]
with coefficients $[\cdot,\cdot]_i\colon V\times V\to V$ and $\alpha_i\colon V\to V$, such that $(V[[t]],[\cdot,\cdot]_t,\alpha_t)$ is a Hom--Lie algebra over the ring of formal power series $\mathbb K[[t]]$. The first--order term $([\cdot,\cdot]_1,\alpha_1)$ is a $2$--Hom--cocycle for the cohomology theory introduced in \cite[\S 5]{MakSil-Forum}.

In analogy with classical deformation theory for Lie algebras, an infinitesimal deformation is obtained by truncating modulo $t^2$. Concretely, one works over the ring $\mathbb K[t]/(t^2)$ and considers
\[
[\cdot,\cdot]_t = [\cdot,\cdot]_0 + t[\cdot,\cdot]_1,
\qquad
\alpha_t = \alpha_0 + t\alpha_1,
\]
subject to the Hom--Jacobi identity in $\mathbb K[t]/(t^2)$. In this situation we shall refer to $([\cdot,\cdot]_1,\alpha_1)$ as the \textbf{infinitesimal} part of the deformation.

The simple Lie algebra $\mathfrak{sl}_2(\mathbb K)$ is classically rigid as a Lie algebra: every formal Lie algebra deformation of $\mathfrak{sl}_2(\mathbb K)$ is equivalent to the trivial one, a fact which reflects the vanishing of $H^2(\mathfrak{sl}_2(\mathbb K),\mathfrak{sl}_2(\mathbb K))$ (Whitehead's lemma). In \cite[\S 6]{MakSil-Forum}, Makhlouf and Silvestrov showed that, in the larger category of Hom--Lie algebras, one can nevertheless construct many non-trivial Hom--Lie deformations of $\mathfrak{sl}_2(\mathbb K)$, including Jackson-type deformations of $\mathfrak{sl}_2(\mathbb K)$ and $q$--deformations of Witt type. They also described all linear maps
\[
\alpha\colon \mathfrak{sl}_2(\mathbb K)\longrightarrow \mathfrak{sl}_2(\mathbb K)
\]
for which the fixed classical brackets
\begin{equation}\label{eq:sl2-bracket}
[x_1,x_2]_0 = 2x_2,\qquad
[x_1,x_3]_0 = -2x_3,\qquad
[x_2,x_3]_0 = x_1
\end{equation}
endow $(\mathfrak{sl}_2(\mathbb K),[\cdot,\cdot]_0,\alpha)$ with a Hom--Lie algebra structure; see \cite[Proposition~6.1]{MakSil-Forum}.

In the same section, they considered infinitesimal Hom--Lie deformations of $\mathfrak{sl}_2(\mathbb K)$, in the above sense, of the form
\[
[\cdot,\cdot]_t = [\cdot,\cdot]_0 + t[\cdot,\cdot]_1,
\qquad
\alpha_t = \alpha_0 + t\alpha_1,
\]
with $\alpha_0 = \mathrm{id}_V$, such that $(V\otimes_{\mathbb K} \mathbb K[t]/(t^2),[\cdot,\cdot]_t,\alpha_t)$ is a Hom--Lie algebra and $([\cdot,\cdot]_1,\alpha_1)$ is a $2$--Hom--cocycle for the Hom--Lie cohomology. They exhibited several explicit families of such pairs $([\cdot,\cdot]_1,\alpha_1)$. For these particular families they observed that, after a suitable choice of parameters, one may arrange $\alpha_1=\mathrm{id}_V$, so that $\alpha_t=(1+t)\,\mathrm{id}_V$ and the Hom--Jacobi identity reduces to the usual Jacobi identity. In other words, all these deformations actually give Lie algebras.

Makhlouf and Silvestrov then carried out further computations with a computer algebra system to construct additional infinitesimal Hom--Lie deformations of $\mathfrak{sl}_2(\mathbb K)$. In doing so they imposed an extra condition: they required that the first--order twisting $(V,[\cdot,\cdot]_0,\alpha_1)$ itself be a Hom--Lie algebra. Under this additional restriction they found that every example produced by the computer again yielded a Lie algebra. On the basis of this evidence they formulated the following conjecture in \cite[\S 6.1, Remark~6.3]{MakSil-Forum}.

\begin{conjecture}[Makhlouf--Silvestrov]
Let $(V,[\cdot,\cdot]_0)\cong\mathfrak{sl}_2(\mathbb K)$ and set $\alpha_0=\mathrm{id}_V$. Suppose that
\[
[\cdot,\cdot]_t = [\cdot,\cdot]_0 + t[\cdot,\cdot]_1,
\qquad
\alpha_t = \alpha_0 + t\alpha_1
\]
define an infinitesimal Hom-Lie deformation over $\mathbb K[t]/(t^2)$, that is to say, $(V\otimes_{\mathbb K} \mathbb K[t]/(t^2),[\cdot,\cdot]_t,\alpha_t)$ is a Hom--Lie algebra. 

Assume in addition that $(V,[\cdot,\cdot]_0,\alpha_1)$ is a Hom--Lie algebra. Then the resulting Hom--Lie algebra $(V,[\cdot,\cdot]_t,\alpha_t)$ is in fact a Lie algebra, in the sense that the bracket $[\cdot,\cdot]_t$ satisfies the ordinary Jacobi identity.
\end{conjecture}

Since $\alpha_t$ is invertible in $\mathbb K[t]/(t^2)$ (being a deformation of $\mathrm{id}_V$), if $[\cdot,\cdot]_t$ is a Lie bracket then one can \textit{untwist} the Hom–Lie structure by setting
\[
\{x,y\}_t := \alpha_t^{-1}([x,y]_t),
\]
which defines a Lie algebra structure $(V,\{\cdot,\cdot\}_t)$ over $\mathbb K[t]/(t^2)$. Conversely, starting from a Lie algebra $(V,\{\cdot,\cdot\}_t)$ and an algebra endomorphism $\alpha_t$, the Yau twisting
\(
[x,y]_t := \alpha_t(\{x,y\}_t)
\)
produces a Hom–Lie algebra. In this sense, the condition that $[\cdot,\cdot]_t$ is a Lie bracket means that $(V,[\cdot,\cdot]_t,\alpha_t)$ comes from a usual Lie algebra by Yau twisting.

The purpose of this note is to give a direct proof of the Makhlouf--Silvestrov conjecture by an explicit computation in a fixed basis of $\mathfrak{sl}_2(\mathbb K)$. Our main result can be stated as follows.

\begin{theorem}\label{thm:main}
Let $(V,[\cdot,\cdot]_0)\cong\mathfrak{sl}_2(\mathbb K)$. Fix a basis $x_1,x_2,x_3$ of $V$ and define $[\cdot,\cdot]_0$ by \eqref{eq:sl2-bracket}. Consider
\[
[\cdot,\cdot]_t = [\cdot,\cdot]_0 + t[\cdot,\cdot]_1,
\qquad
\alpha_t = \alpha_0 + t\alpha_1,
\qquad
\alpha_0=\mathrm{id}_V,
\]
where $[\cdot,\cdot]_1\colon V\times V\to V$ is bilinear and skew--symmetric and $\alpha_1\colon V\to V$ is linear. Assume that $([\cdot,\cdot]_t,\alpha_t)$ defines an infinitesimal Hom--Lie deformation of $(V,[\cdot,\cdot]_0)$ over $\mathbb K[t]/(t^2)$ and that $(V,[\cdot,\cdot]_0,\alpha_1)$ is a Hom--Lie algebra.

Then the bracket $[\cdot,\cdot]_t$ satisfies the usual Jacobi identity over $\mathbb K[t]$, and hence $(V,[\cdot,\cdot]_t)$ is a Lie algebra for every value of $t$. 

\end{theorem}

The proof, given in the next section, is entirely elementary. We fix a basis of $\mathfrak{sl}_2(\mathbb K)$, write down the most general first--order perturbation $[\cdot,\cdot]_1$ and the most general twisting map $\alpha_1$ satisfying the condition that $(V,[\cdot,\cdot]_0,\alpha_1)$ is a Hom--Lie algebra, and then expand both the Hom--Jacobi identity for $([\cdot,\cdot]_t,\alpha_t)$ and the usual Jacobi identity for $[\cdot,\cdot]_t$. Under the additional constraint on $\alpha_1$, we obtain exactly the same first--order relations for the structure constants of $[\cdot,\cdot]_1$ from the two identities, and these relations already force the second--order part of the Jacobiator of $[\cdot,\cdot]_t$ to vanish. This implies that $[\cdot,\cdot]_t$ is a Lie bracket.
\section{Proof of the Makhlouf--Silvestrov conjecture}

In this section, we prove Theorem~\ref{thm:main}. Throughout, we fix a field
$\mathbb{K}$ of characteristic~$0$ and a $3$--dimensional $\mathbb{K}$--vector space $V$ with
basis $(x_1,x_2,x_3)$, endowed with the Lie bracket $[\cdot,\cdot]_0$
given by \eqref{eq:sl2-bracket}. Thus, $(V,[\cdot,\cdot]_0)$ is identified
with $\mathfrak{sl}_2(\mathbb{K})$.

\subsection{The constraint that $(V,[\cdot,\cdot]_0,\alpha_1)$ is Hom--Lie}

Let $\alpha_1\colon V\to V$ be a linear endomorphism. We write its matrix form
with respect to the chosen basis as
\begin{align}\label{eq:alpha1-matrix}
\alpha_1(x_1) &= a_{11}x_1 + a_{21}x_2 + a_{31}x_3,\\
\alpha_1(x_2) &= a_{12}x_1 + a_{22}x_2 + a_{32}x_3,\\
\alpha_1(x_3) &= a_{13}x_1 + a_{23}x_2 + a_{33}x_3,
\end{align}
so that $(a_{ij})_{1\leq i,j\leq 3}$ is the matrix of $\alpha_1$ in the
basis $(x_1,x_2,x_3)$.

The condition that $(V,[\cdot,\cdot]_0,\alpha_1)$ is itself a Hom--Lie
algebra can be expressed as the Hom--Jacobi identity
\[
\sum_{\circlearrowleft x,y,z}
[\alpha_1(x),[y,z]_0]_0 = 0
\qquad \text{for all }x,y,z\in V.
\]
By trilinearity and skew--symmetry, it suffices to impose this identity on
the triple $(x_1,x_2,x_3)$. Writing the three cyclic summands
\[
[\alpha_1(x_1),[x_2,x_3]_0]_0,\qquad
[\alpha_1(x_2),[x_3,x_1]_0]_0,\qquad
[\alpha_1(x_3),[x_1,x_2]_0]_0
\]
and using \eqref{eq:sl2-bracket} one checks that
\[
\sum_{\circlearrowleft}
[\alpha_1(x_i),[x_j,x_k]_0]_0
=
(2a_{22}-2a_{33})x_1 + (4a_{13}-2a_{21})x_2 + (-4a_{12}+2a_{31})x_3,
\]
where the sum is taken over the cyclic permutations $(i,j,k)$ of
$(1,2,3)$. Hence $(V,[\cdot,\cdot]_0,\alpha_1)$ is Hom--Lie if and only if
\begin{equation}\label{eq:H}
2a_{22}-2a_{33} = 0,\qquad
4a_{13}-2a_{21} = 0,\qquad
-4a_{12}+2a_{31} = 0.
\end{equation}
Equivalently,
\[
a_{22}=a_{33},\qquad a_{21}=2a_{13},\qquad a_{31}=2a_{12}.
\]
Thus $\alpha_1$ is determined by the six free parameters
$a_{11},a_{12},a_{13},a_{22},a_{23},a_{32}\in\mathbb{K}$, with
$a_{21},a_{31},a_{33}$ constrained by \eqref{eq:H}. This recovers the
$6$--parameter family described in \cite[Proposition~6.1]{MakSil-Forum},
written in slightly different coordinates.

\subsection{First--order Hom--Lie deformations}

We now consider an infinitesimal Hom--Lie deformation of
$(V,[\cdot,\cdot]_0,\alpha_0)$ over the ring $\mathbb{K}[t]/(t^2)$ of the form
\[
[\cdot,\cdot]_t = [\cdot,\cdot]_0 + t[\cdot,\cdot]_1,\qquad
\alpha_t = \mathrm{id}_V + t\alpha_1,
\]
where $[\cdot,\cdot]_1\colon V\times V\to V$ is bilinear and skew--symmetric.
We can write the most general such bracket as
\begin{equation}\label{eq:br1}
\begin{aligned}
[x_1,x_2]_1 &= p_1 x_1 + p_2 x_2 + p_3 x_3,\\
[x_1,x_3]_1 &= q_1 x_1 + q_2 x_2 + q_3 x_3,\\
[x_2,x_3]_1 &= r_1 x_1 + r_2 x_2 + r_3 x_3,
\end{aligned}
\end{equation}
with $p_i,q_i,r_i\in\mathbb{K}$, and the remaining brackets determined by
skew--symmetry.

The Hom--Jacobi identity for the deformed structure
$(V\otimes_{\mathbb{K}} \mathbb{K}[t]/(t^2),[\cdot,\cdot]_t,\alpha_t)$ reads
\begin{equation}\label{eq:HomJacobi-t}
\sum_{\circlearrowleft x,y,z}
[\alpha_t(x),[y,z]_t]_t = 0
\qquad\text{in }\mathbb{K}[t]/(t^2)
\end{equation}
for all $x,y,z\in V$. Again, it is enough to consider the triple
$(x_1,x_2,x_3)$.

Expanding the left--hand side of \eqref{eq:HomJacobi-t} in powers of $t$,
we obtain
\[
\sum_{\circlearrowleft}
[\alpha_t(x_i),[x_j,x_k]_t]_t
=
J_0 + t J_1 + t^2 J_2+ t^3 J_3,
\]
where $J_0,J_1,J_2, J_3\in V$. The constant term $J_0$ vanishes because
$(V,[\cdot,\cdot]_0)$ is a Lie algebra and $\alpha_0=\mathrm{id}_V$.
Since we are working in $\mathbb{K}[t]/(t^2)$, the Hom--Jacobi identity is
equivalent to the condition $J_1=0$.

A direct computation using \eqref{eq:sl2-bracket},
\eqref{eq:alpha1-matrix} and \eqref{eq:br1} gives
\[
J_1 =
\bigl(2a_{22}-2a_{33}-p_2-q_3\bigr)x_1
+\bigl(4a_{13}-2a_{21}+2q_1+2r_2\bigr)x_2
+\bigl(-4a_{12}+2a_{31}+2p_1-2r_3\bigr)x_3.
\]
Thus the infinitesimal Hom--Jacobi condition $J_1=0$ is equivalent to
\begin{equation}\label{eq:D}
\begin{aligned}
2a_{22}-2a_{33} &= p_2+q_3,\\
4a_{13}-2a_{21} &= -2(q_1+r_2),\\
-4a_{12}+2a_{31} &= -2(p_1-r_3).
\end{aligned}
\end{equation}

Combining \eqref{eq:D} with the relations \eqref{eq:H} expressing that
$(V,[\cdot,\cdot]_0,\alpha_1)$ is Hom--Lie, we obtain equations involving
only the structure constants of $[\cdot,\cdot]_1$.

\begin{lemma}\label{lem:E}
Assume that $(V,[\cdot,\cdot]_0,\alpha_1)$ is a Hom--Lie algebra, that is,
assume that \eqref{eq:H} holds.
Then the infinitesimal Hom--Jacobi condition \eqref{eq:D} is equivalent to
\begin{equation}\label{eq:E}
p_2 + q_3 = 0,\qquad
q_1 + r_2 = 0,\qquad
p_1 - r_3 = 0.
\end{equation}
\end{lemma}

\begin{proof}
Subtracting \eqref{eq:H} from \eqref{eq:D} term by term yields
\eqref{eq:E} immediately. Conversely, if \eqref{eq:H} and \eqref{eq:E}
hold, then \eqref{eq:D} follows.
\end{proof}

Thus, under the additional constraint on $\alpha_1$, the infinitesimal
Hom--Jacobi condition is encoded entirely in the three linear equations
\eqref{eq:E} for the first--order bracket $[\cdot,\cdot]_1$.

\subsection{The ordinary Jacobi identity for $[\cdot,\cdot]_t$}

We now turn to the usual Jacobi identity for the bracket
\[
[\cdot,\cdot]_t = [\cdot,\cdot]_0 + t[\cdot,\cdot]_1
\]
on $V$. Define the Jacobiator of $[\cdot,\cdot]_t$ by
\[
K_t(x,y,z)
=
[x,[y,z]_t]_t + [y,[z,x]_t]_t + [z,[x,y]_t]_t.
\]
Once again it suffices to evaluate $K_t$ on the triple
$(x_1,x_2,x_3)$. A direct calculation using \eqref{eq:sl2-bracket} and
\eqref{eq:br1} shows that
\[
K_t(x_1,x_2,x_3)
= tK_1 + t^2 K_2,
\]
where
\[
K_1 =
(-p_2-q_3)x_1 + (2q_1+2r_2)x_2 + (2p_1-2r_3)x_3
\]
and
\[
\begin{aligned}
K_2 &= (p_1r_2-p_2r_1+q_1r_3-q_3r_1)x_1\\
&\quad +(-p_1q_2+p_2q_1+q_2r_3-q_3r_2)x_2\\
&\quad +(-p_1q_3-p_2r_3+p_3q_1+p_3r_2)x_3.
\end{aligned}
\]

The usual Jacobi identity for $[\cdot,\cdot]_t$ is equivalent to the
condition
\[
K_t(x_1,x_2,x_3) = 0
\quad\text{for all }t,
\]
which in turn is equivalent to the simultaneous vanishing of $K_1$ and
$K_2$.

We now relate these conditions to \eqref{eq:E}.

\begin{lemma}\label{lem:K1}
The equality $K_1=0$ is equivalent to the system \eqref{eq:E}.
\end{lemma}

\begin{proof}
The coefficients of $x_1,x_2,x_3$ in $K_1$ are $-p_2-q_3$,
$2q_1+2r_2$ and $2p_1-2r_3$, respectively. Since the characteristic of
$\mathbb{K}$ is zero, these coefficients vanish if and only if \eqref{eq:E} holds.
\end{proof}

\begin{lemma}\label{lem:K2}
Assume that \eqref{eq:E} holds. Then $K_2=0$. Equivalently, under the
relations \eqref{eq:E}, the usual Jacobi identity for $[\cdot,\cdot]_t$
holds for all $t$.
\end{lemma}

\begin{proof}
From \eqref{eq:E} we have
\[
p_2=-q_3,\qquad r_2=-q_1,\qquad r_3=p_1.
\]
Substituting these relations into the coefficients of $K_2$ we obtain:
\begin{itemize}
\item The coefficient of $x_1$ is
\[
p_1r_2-p_2r_1+q_1r_3-q_3r_1
= p_1(-q_1) - p_2r_1 + q_1p_1 - q_3r_1
= -(p_2+q_3)r_1 = 0.
\]
\item The coefficient of $x_2$ is
\[
\begin{aligned}
-p_1q_2+p_2q_1+q_2r_3-q_3r_2
&= -p_1q_2+p_2q_1+q_2p_1-q_3(-q_1)\\
&= (p_2+q_3)q_1 = 0.
\end{aligned}
\]
\item The coefficient of $x_3$ is
\[
\begin{aligned}
-p_1q_3-p_2r_3+p_3q_1+p_3r_2
&= -p_1q_3-p_2p_1+p_3q_1+p_3(-q_1)\\
&= -p_1(q_3+p_2) = 0.
\end{aligned}
\]
\end{itemize}
Thus all three coefficients vanish, and hence $K_2=0$.
\end{proof}

We can now complete the proof of our main theorem.

\begin{proof}[Proof of Theorem~\ref{thm:main}]
By hypothesis, $(V,[\cdot,\cdot]_0,\alpha_1)$ is a Hom--Lie algebra, so
\eqref{eq:H} holds. Since $([\cdot,\cdot]_t,\alpha_t)$ is an infinitesimal
Hom--Lie deformation, the Hom--Jacobi identity \eqref{eq:HomJacobi-t}
holds modulo $t^2$, and hence \eqref{eq:D} is satisfied. By
Lemma~\ref{lem:E} this is equivalent to the linear relations
\eqref{eq:E} on the structure constants of $[\cdot,\cdot]_1$.

Lemma~\ref{lem:K1} shows that \eqref{eq:E} is equivalent to the vanishing
of the coefficient $K_1$ of $t$ in the Jacobiator
$K_t(x_1,x_2,x_3)$, and Lemma~\ref{lem:K2} shows that the same relations
force the coefficient $K_2$ of $t^2$ to vanish as well. Thus
$K_t(x_1,x_2,x_3)=0$ for all $t$. By trilinearity and skew--symmetry this
implies that the Jacobi identity holds for all triples $(x,y,z)\in V^3$.

Therefore $[\cdot,\cdot]_t$ is a Lie bracket on $V$ for every value of
$t$, and the associated Hom--Lie algebras $(V,[\cdot,\cdot]_t,\alpha_t)$
are Lie algebras in the sense of the Makhlouf--Silvestrov conjecture.
\end{proof}


\begin{corollary}[Makhlouf-Silvestrov conjecture]\label{cor:MS-conjecture}
Every infinitesimal Hom--Lie deformation
\[
[\cdot,\cdot]_t = [\cdot,\cdot]_0 + t[\cdot,\cdot]_1,
\qquad
\alpha_t = \mathrm{id}_V + t\alpha_1
\]
of $(\mathfrak{sl}_2(\mathbb K),[\cdot,\cdot]_0)$ over $\mathbb K[t]/(t^2)$ such that $(\mathfrak{sl}_2(\mathbb K),[\cdot,\cdot]_0,\alpha_1)$ is a Hom--Lie algebra is in fact a Lie algebra, in the sense that the bracket $[\cdot,\cdot]_t$ satisfies the ordinary Jacobi identity.
\end{corollary}

\begin{proof}
This is an immediate consequence of Theorem~\ref{thm:main}, applied with $V \cong \mathfrak{sl}_2(\mathbb K)$ and $\alpha_0=\mathrm{id}_V$.
\end{proof}

\end{document}